\documentclass[10pt,a4paper]{amsart}
\usepackage{latexsym}
\usepackage{amssymb} 
\usepackage{enumerate} 
\usepackage{ifpdf}

\usepackage{verbatim} 
\usepackage{mathrsfs} 
\usepackage{multirow} 
\usepackage{srcltx} 
\ifpdf 
 \usepackage[pdftex]{graphicx} 
 \DeclareGraphicsExtensions{.pdf} 
 \usepackage{hyperref} 
\else 
 \usepackage{graphicx} 
 \DeclareGraphicsExtensions{.eps} 
 \usepackage{hyperref} 
\fi 
 
\newcommand\rightmap[1]{\smash{\mathop{\ \rightarrow\ }\limits^{#1}}} 
\newcommand\eqmap[1]{\smash{\mathop{=}\limits^{#1}}} 
\newcommand\mapstomap[1]{\smash{\mathop{\ \mapsto\ }\limits^{#1}}}
\newcommand\dotsmap[1]{\smash{\mathop{\ \dots\ }\limits^{#1}}}
\newcommand\injmap[1]{\smash{\mathop{\ \hookrightarrow\ }\limits^{#1}}}
\newcommand\downmap[1]{\downarrow\rlap{$\vcenter{\hbox{$\scriptstyle#1$}}$}}

\newcommand{\cD}{\mathcal D} 
\newcommand{\cB}{\mathcal B} 
\newcommand{\cC}{\mathcal C} 
\newcommand{\cX}{\mathcal X}

\newcommand{\cF}{\mathcal F}

\newcommand{\cL}{\mathcal L} 
\newcommand{\cM}{\mathcal M}

\newcommand{\cR}{\mathcal R} 
\newcommand{\cQ}{\mathcal Q}

\newcommand{\e}{\varepsilon}

\newcommand{\Hom}{\text{\rm Hom}} 
\newcommand{\orb}{\text{\rm orb}} 
\newcommand{\Sing}{\text{\rm Sing}}

\DeclareMathOperator{\spec}{Spec}
\newcommand{\Ceva}{\text{\rm CEVA}}
\DeclareMathOperator{\ab}{\bf ab}

\newcommand{\ZZ}{\mathbb Z} 
\newcommand{\CC}{\mathbb C}

\newcommand{\PP}{\mathbb P}

\newcommand{\TT}{\mathbb T} 
\newcommand{\FF}{\mathbb F}

\def\Char {{\rm Char}} 
\def\Spec {{\rm Spec}} 
\def\Supp {{\rm Supp}}

\def\dim{{\rm dim}}

\def\deg{{\rm \,deg\, }}

\def\cal{\mathcal} 
\def\SC{{augmented Ceva}}
 
\newtheorem{theorem}{Theorem}[section] 
\newtheorem{lemma}[theorem]{Lemma} 
\newtheorem{prop}[theorem]{Proposition} 
\newtheorem{cor}[theorem]{Corollary} 
\theoremstyle{definition} 
\newtheorem{dfn}[theorem]{Definition}

\newtheorem{example}[theorem]{Example} 
 
\theoremstyle{remark} 
\newtheorem{remark}[theorem]{Remark} 

\newcommand\enet[1]{\renewcommand\theenumi{#1}
\renewcommand\labelenumi{\theenumi}}

\begin{document} 
 
\title{Depth of characters of curve complements and orbifold pencils}
\author{E. Artal Bartolo, J.I.~Cogolludo-Agust{\'\i}n and A.~Libgober} 


\address{Departamento de Matem\'aticas, IUMA\\ 
Universidad de Zaragoza\\ 
C.~Pedro Cerbuna 12\\ 
50009 Zaragoza, Spain} 
\email{artal@unizar.es,jicogo@unizar.es} 
 
\address{ 
Department of Mathematics\\ 
University of Illinois\\ 
851 S.~Morgan Str.\\ 
Chicago, IL 60607} 
\email{libgober@math.uic.edu} 

\thanks{Partially supported by the Spanish Ministry of
Education MTM2010-21740-C02-02. 
The third named author was also partially supported by an NSF grant.} 
 


\begin{abstract}
The present work is a user's guide to the results of~\cite{acl-depth},
where a description of the space of characters of a quasi-projective 
variety was given in terms of global quotient orbifold pencils.

Below we consider the case of plane curve complements and hyperplane arrangements.
In particular, an infinite family of curves exhibiting characters of any torsion and 
depth~3 will be discussed. Also, in the context of line arrangements, it will be shown 
how geometric tools, such as the existence of orbifold pencils, can replace the group 
theoretical computations via fundamental groups when studying characters of finite order, 
specially order two. Finally, we revisit an Alexander-equivalent Zariski pair considered 
in the literature and show how the existence of such pencils distinguishes both curves.
\end{abstract}

\maketitle

\section{Introduction}
Let $\cal X$ be the complement of a reduced (possibly reducible) 
projective curve $\cD$ in the complex projective plane $\PP^2$. 
The space of characters of the fundamental group 
$\Char(\cX)=\Hom(\pi_1(\cX),\CC^*)$ has an interesting 
stratification by subspaces, given by the cohomology of 
the rank one local system associated 
with the character:
\begin{equation}\label{jump}
V_k(\cX):=\{\chi \mid \dim H^1(\cal X,\chi) = k \}.
\end{equation}
The closures of these jumping loci 
in $\Char(\cX)$ 
were called in~\cite{charvar} the characteristic varieties of $\cX$. 
More precisely, the 
characteristic varieties associated to $\cX$  were defined in~\cite{charvar}
as the zero sets of Fitting ideals of the $\CC[\pi_1/\pi_1']$-module 
which is the complexification $\pi_1'/\pi_1'' \otimes \CC$ 
of the abelianized commutator of the fundamental 
group $\pi_1(\cX)$ (cf. section~\ref{prelim} for more details). 
In particular the characteristic varieties (unlike the jumping sets 
of the cohomology dimension greater than one) 
depend only on the fundamental group. Fox calculus provides
an effective method for calculating the characteristic varieties if a presentation of the fundamental 
group by generators and relators is known. 

For each character $\chi \in \Hom(\pi_1(\cal X),\CC^*)$ the 
\emph{depth} 
was defined in~\cite{charvar} as
\begin{equation}
d(\chi):=\dim H^1(\cal X,\chi)
\end{equation} 
so that the strata~\eqref{jump} are the sets on which $d(\chi)$ is constant.

In~\cite{acl-depth}, we describe a geometric interpretation  
of the depth by relating it to the pencils on $\cX$ i.e. 
holomorphic maps $\cX \rightarrow C, \ {\rm dim} C=1$  
having multiple fibers.
In fact the discussion in~\cite{acl-depth} is in a more 
general context in which $\cX$ is a smooth quasi-projective variety.
\footnote{Much of the discussion in the first two sections below 
applies to general quasi-projective varieties 
(cf.~\cite{acl-depth}), but in the present paper we will stay in the hypersurface complement context.
Note that the characteristic varieties only depend on the fundamental group, hence, by the Lefschetz-type 
theorems it is enough to consider the curve complement class.}
The viewpoint of~\cite{acl-depth} (and~\cite{acm-charvar-orbifolds})
is that such a pencil can be considered as a map in the category of orbifolds.
The orbifold structure of the curve $C$ is matched by the 
structure of multiple fibers of the pencil.
The main result of \cite{acl-depth} can be stated as follows: 

\begin{theorem} 
\label{thm-main}
Let $(\cX,\chi)$ be a pair consisting of a smooth quasi-projective 
manifold $\cX$, whose smooth compactification $\bar \cX$ 
satisfies $H^1(\bar \cX,\CC)=0$, and let $\chi$ be a character 
of the fundamental group of $\cX$. Assume that the depth of $\chi$ 
is positive.
Then there is a (possibly non-compact)
orbifold curve $\cC$, a character $\rho$ of
its orbifold fundamental group $\pi_1^{orb}(\cC)$ and 
$n$ strongly independent orbifold pencils $f_i: \cX \rightarrow \cC$ such 
that $f_i^*(\rho)=\chi$.
Moreover,  
\begin{enumerate}
\enet{\rm(\arabic{enumi})}
\item\label{thm-main-part1}
$d(\chi) \ge n d(\rho)$.

\item\label{thm-main-part2}
If $\cC$ is a global quotient orbifold of a one-dimensional algebraic group, then the inequality above
is in fact an equality, $d(\chi) = n d(\rho)$.
\end{enumerate}

Vice versa, given an orbifold pencil $f: \cX \rightarrow \cC$ and a character
$\rho \in \Char \pi_1^{orb}(\cC)$ with a positive depth then 
$\chi=f^*(\rho)$ has a positive depth at least as large as $d(\rho)$.
\end{theorem}

We refer to section~\ref{prelim} for all the required definitions,
and in particular, the definition of strongly independent pencils.

According to this result, given a character $\chi$ with a positive depth, one 
automatically has an orbifold pencil with the 
depth $d(\chi)$ being bounded below by the constant specified in
Theorem~\ref{thm-main}\ref{thm-main-part1} by the geometric data.
One can compare this statement with the numerous previous results 
on existence of pencils on quasi-projective manifolds which 
mostly guarantee existence of ordinary pencils corresponding 
to the components of~\eqref{jump} having a positive dimension
(cf.~\cite{simp,arapura} and references therein).
For example, if $\chi$ has a positive depth and an infinite order, then it 
must belong to a component of characteristic variety 
of positive dimension since the isolated points of characteristic 
variety have a finite order by~\cite{nonvanishing}. Hence
the results of~\cite{arapura} can be applied to such a component to obtain 
a pencil $f: \cX \rightarrow \cC$ and a character $\rho \in \Char \pi_1(\cC)$
such that $\chi=f^*(\rho)$. Here $\cC$ is the 
complement in $\PP^1$ to a finite set containing say $d>2$ points.
Moreover, the number of independent pencils in the sense of section~\ref{prelim} 
is equal to one (cf. Remark~\ref{rem-non-orbifold}) and  
the depth of $\rho \in \Char \pi_1(\cC)$ is equal to $d-2$). 
Since it was show in \cite{arapura} that 
${\rm dim} H^1(\cX,\chi) \ge {\rm dim} H^1(\cC,\rho)$
the inequality in Theorem~\ref{thm-main}\ref{thm-main-part1}
follows from the latter. Note that the orbifold
pencils, which Theorem~\ref{thm-main} claims, \emph{without the orbifold structure} 
are just rational pencils and the connection with the jumping loci disappears. 

The goal of this paper, is to illustrate in detail 
both parts~\ref{thm-main-part1} and~\ref{thm-main-part2} of Theorem~\ref{thm-main} with examples 
in which orbifolds are unavoidable. We start with a section 
reviewing mainly known results on the cohomology of local systems, 
characteristic varieties, orbifolds, and Zariski pairs making
possible to read the rest of the paper unless one is interested 
in the proofs of mentioned results.
Then, firstly, in section~\ref{sec-fermat}, a family of curves is considered 
for which the characteristic variety containing isolated characters 
having torsion or arbitrary finite order 
and whose depth is~3. The calculations illustrate the use 
of Fox calculus for finding explicit form description 
of the characteristic varieties. Secondly, in the context of line arrangements,
examples of Ceva and \SC\ arrangements are considered in section~\ref{sec-ceva}. Their characteristic 
varieties have been studied in the literature via computer aided calculations based on fundamental group 
presentations and Fox calculus. Here we present an alternative way to study such varieties independent of 
the fundamental group illustrating the geometric approach of Theorem~\ref{thm-main}. 
Finally, in section~\ref{sec-zariski-pair} we discuss on a Zariski pair of sextic
curves, whose Alexander polynomials coincide. We determine this Zariski pair
by the existence of orbifold pencils.

\subsection*{Acknowledgments}
The authors want to express thanks to the organizers of the 
\emph{Intensive research period on Configuration Spaces: Geometry, Combinatorics and Topology,} 
at the \emph{Centro di Ricerca Matematica Ennio De Giorgi}, Pisa, in May 2010   
for their hospitality and excellent working atmosphere. 
The second and third named authors want to thank the University of Illinois and Zaragoza respectively,
where part of the work on the material of this paper was done during their visit in Spring and Summer 2011.

\section{Preliminaries}\label{prelim}
\label{sec-preliminaries}

In this section the necessary definitions 
used in Theorem~\ref{thm-main} will be reviewed
together with material on 
the characteristic varieties and Zariski pairs with the 
aim to keep the discussion of the upcoming sections in 
a reasonably self-contained manner.

\subsection{Characteristic varieties}
\label{libg-var}
\mbox{}

Characteristic varieties appeared first in the literature in the context of algebraic curves 
in~\cite{Libgober-homology}. They can be defined as follows.

Let $\cD:=\cD_1\cup\dots \cup \cD_r$ be the decomposition of 
a reduced curve $\cD$ into irreducible components
and let  $d_i:=\deg \cD_i$ denote the degrees of 
the components $\cD_i$. Let $\tau:=\gcd(d_1,\dots,d_r)$ 
and $\cX=\PP^2 \setminus\cD$. Then (cf.~\cite{charvar})
\begin{equation}
\label{eq-h1}
H_1(\cX;\ZZ)={\left\langle \bigoplus_{i=1}^r \gamma_i \ZZ \right\rangle}/
{\langle d_1 \gamma_1+\dots+d_r \gamma_r\rangle}\approx 
\ZZ^{r-1}\oplus {\ZZ}/{\tau \ZZ},
\end{equation}
where $\gamma_i$ is the homology class of a meridian of $\cD_i$
 (i.e. the boundary of small disk transversal to $\cD_i$ at 
a smooth point).

Let $\ab:G:=\pi_1(\cX)\to H_1(\cX;\ZZ)$ be epimorphism of abelianization.
The kernel $G'$ of $\ab$, i.e. the commutator of $G$, 
defines the universal Abelian covering of $\cX$, say
$\cX_{\ab} \rightmap{\pi} \cX$, whose group of deck transformations is $H_1(\cX;\ZZ)=G/G'$.
This group of deck transformations, since it acts on $\cX_{\ab}$, also 
acts on $H_1(\cX_{\ab};\ZZ)=G'/G''$. 
\footnote{this action corresponds to the action of $G/G'$ on $G'/G''$ by conjugation} 
This allows to endow $M_{\cD,\ab}:=H_1(\cX_{\ab};\ZZ)\otimes \CC$ 
(as well as $\tilde M_{\cD,\ab}:=H_1(\cX_{\ab},\pi^{-1}(*);\ZZ)\otimes \CC$)
with a structure of $\Lambda_\cD$-module where 
\begin{equation}
\label{eq-ring}
\Lambda_\cD:=\CC[G/G']\approx \CC[t_1^{\pm 1},\dots,t_r^{\pm 1}]/(t_1^{d_1}\cdot\ldots\cdot t_r^{d_r}-1).
\end{equation}

Note that $Spec \Lambda_\cD$ can be identified with the commutative 
affine algebraic group $\Char \pi_1(X)$ having $\tau$ tori ${(\CC^*)}^{r-1}$ 
as connected components. Indeed, the elements of $\Lambda_{\cD}$ can be 
viewed as the functions on the group of characters of $G$.


Since $G$ is a finitely generated group, the module $M_{\cD,\ab}$ 
(resp. $\tilde M_{\cD,\ab}$) is a finitely generated $\Lambda_\cD$-module: 
\footnote{in most interesting examples with non-cyclic $G/G'$ the group
$G'/G''$ is infinitely generated.}
in fact one can construct a presentation of $M_{\cD,\ab}$ 
(resp. $\tilde M_{\cD,\ab}$) with the number of $\Lambda_{\cD}$-generators 
at most $n \choose 2$ (resp. $n$), where $n$ is the number of generators of $G$.
If $G/G'$ is not cyclic (i.e. $r>2$ or $r\geq 2$ and $\tau>1$) then
$\Lambda_{\cD}$ is not a Principal Ideal Domain. One way to approach the 
$\Lambda_{\cD}$-module structure of both $M_{\cD,\ab}$ and $\tilde M_{\cD,\ab}$ 
is to study their Fitting ideals (cf.~\cite{eisenbud}).

Let us briefly recall the relevant definitions. Let $R$ be a commutative Noetherian 
ring with unity and $M$ a finitely generated $R$-module. Choose a finite free presentation 
for $M$, say $\phi: R^m \rightarrow R^n$, where $M=\text{coker } \phi.$ The homomorphism $\phi$ 
has an associated $(n \times m)$ matrix $A_{\phi}$ with coefficients in $R$ such that
$\phi(x)=A_{\phi} x$ (the vectors below are represented as the column matrices).
\begin{dfn}
\label{def-fitting} 
The $k$-th \emph{Fitting ideal} $F_k(M)$ of $M$ is defined as the ideal generated by
$$\begin{cases}
0 & \text{if } k \leq \max\{0,n-m\} \\
1 & \text{if } k > n \\
\text{minors of } A_{\phi} \text{ of order } (n-k+1) & \text{otherwise.}
\end{cases}
$$
It will be denoted $F_k$ if no 
ambiguity seems likely to arise.
\end{dfn}

\begin{dfn}{\cite{Libgober-homology}}
\label{def-charvar}
With the above notations the \emph{$k$-th characteristic variety}  ($k>0$)
of $\cX=\PP^2\setminus \cD$ can be defined as the zero-set of the ideal $F_k(M_{\cD,\ab})$
$$
\Char_k(\cX):=Z(F_k(M_{\cD,\ab}))\subset\Spec \Lambda_{\cD}=\Char \pi_1(\PP^2\setminus \cD).
$$

Then $V_k(\cX)$ is the set of characters in $\Char_k(\cX)$ 
which do not belong to $\Char_j(\cX)$ for $j>k$. 
If a character $\chi$ belongs to $V_k(\cX)$ 
then $k$ is called the {\it depth} of $\chi$ and denoted by $d(\chi)$
(cf. \cite{charvar}). 

An alternative notation for $\Char_k(\PP^2 \setminus \cD)$ is $\Char_{k,\PP}(\cD)$.
\end{dfn}

\begin{remark} Essentially without loss of generality one can 
consider only the cases when 
the quotient by an ideal in the definition of the ring 
$\Lambda_\cD$ in~\eqref{eq-ring} is absent
i.e. consider only the modules of the ring of Laurent polynomials. Indeed, 
consider a line $L$ not contained in $\cD$ and in general position
(i.e. which does not contain singularities of $\cD$ and is transversal to it). 
Then $\Lambda_{L\cup \cC}$ is 
isomorphic to $\CC[t_1^{\pm 1},\dots,t_r^{\pm 1}]$. Moreover, since 
we assume transversality $L\pitchfork\cD$, then
the $\Lambda_{L\cup \cD}$-module $M_{L\cup \cD,\ab}$ does not depend on $L$
(see for instance~\cite[Proposition~1.16]{ji-zariski}).
The characteristic variety $\Char_{k,\PP}(L\cup \cD)$ 
determines $\Char_{k,\PP}(\cD)$ (cf. \cite{charvar,ji-zariski}).
By abuse of language it is called the 
\emph{$k$-th affine characteristic variety} and denoted simply by $\Char_{k}(\cD)$. 


One can also use the module $\tilde M_{\cD,\ab}$ to obtain the characteristic
varieties of $\cD$. One has the following connection
$$
\Char_k(\cX)\setminus \bar 1=Z(F_{k+1}(\tilde M_{\cD,\ab}))\setminus \bar 1,
$$
where $\bar 1$ denotes the trivial character.
\end{remark}

\begin{remark} The depth of a character appears in explicit formulas 
for the first Betti number of cyclic and abelian unbranched 
and branched covering spaces (cf. \cite{Libgober-homology,eko,sakuma})
\end{remark}

\begin{remark}\label{charvargroups}
One can also define the $k$-th characteristic variety $\Char_k(G)$ of 
any finitely generated group $G$ (such that the abelianization $G/G' \ne 0$
or, more generally, for a surjection $G \rightarrow A$ where 
$A$ is an abelian group) 
as the $k$-th characteristic variety of the 
$\Lambda_G=\CC[G/G']$-module $M_G=H_1(\cX_{G,\ab})$
 obtained 
by considering the CW-complex $\cX_G$ 
associated with a presentation of $G$ and its universal 
abelian covering space $\cX_{G,\ab}$ 
(respectively considering the 
covering space of $\cX_G$ associated with the kernel of the map to $A$).
Such invariant is independent of the finite presentation of $G$
(resp. depends only on $G \rightarrow A$). This construction will 
be applied below to the orbifold fundamental groups of one dimensional 
orbifolds.
\end{remark}

\begin{remark}
Note that one has:
\begin{itemize}
\item $\Char_k(\cD)=\Supp_{\Lambda_\cD} \wedge^i(H_1(\cX_{\ab};\CC))$,
\item $\Spec\Lambda_{L\cup \cD}=\TT^r=(\CC^*)^{r}$, for the affine case, and
\item $\Spec\Lambda_\cD=\TT_\cD=\{\omega^i\}_{i=0}^{\tau-1}\times (\CC^*)^{r-1}=
V(t_1^{d_1}\cdot\ldots\cdot t_r^{d_r}-1)\subset\TT^r$,
where $\omega$ is a $\tau$-th primitive root of unity for the curves
in projective plane.
\end{itemize}

Note also that in the case of a finitely presented group $G$ such that 
$G/G'=\ZZ^r\oplus \ZZ/\tau_1\ZZ\oplus\dots\oplus \ZZ/\tau_s\ZZ$ 
one has
\begin{equation}
\Spec {\Lambda_G}=\TT_G=\{(\omega^{i_1}_1,\dots,\omega^{i_s}_s)\mid i_k=0,\dots,\tau_k-1,\ k=1,\dots,s\}
\times (\CC^*)^r,
\end{equation}
where as above
$\Lambda_G=\CC[G/G']$ and $\omega_i$ is a $\tau_i$-th primitive root of unity.
\end{remark}

Let $\cX$ be a smooth quasi-projective variety such that for its smooth compactification 
$\bar \cX$ one has $H^1(\bar \cX,\CC)=0$. This of course includes the cases $\cX=\PP^2\setminus \cD$.
The structure of the closures of the strata $V_k(\cX)$ is given by the following fundamental result.

\begin{theorem}[\cite{arapura}]
\label{arapurath} 
The closure of each $V_k(\cX)$ 
is a finite union of cosets of subgroups of $\Char(\pi_1(\cX))$. 
Moreover, for each irreducible 
component $W$ of $V_k(\cX)$ having a positive dimension  
there is a  pencil $f: \cX \rightarrow C$, where 
$C$ is a $\PP^1$ with deleted points, and a torsion character
$\chi 
\in  \Char_k(\cX)$ 
such that $W=\chi f^*H^1(C,\CC^*)$.
\end{theorem} 

\subsection{Essential Coordinate Components}
\mbox{}

Let $\cD'\subsetneq\cD$ be curve whose components form 
a subset of the set of components of $\cD$. There is a natural epimorphism
$\pi_1(\PP^2\setminus\cD)\twoheadrightarrow\pi_1(\PP^2\setminus\cD')$ induced
by the inclusion. This surjection induces a natural inclusion
$\Spec \Lambda_{\cD'}\subset 
\Spec \Lambda_{\cD}$. With identification of the generators 
of $\Lambda_{\cD}$  with components of $\cD$ as above, 
this embedding is obtained by assigning~$1$
to the coordinates corresponding to those irreducible components of $\cD$ which
are not in $\cD'$ (cf.~\cite{charvar}).

The embedding $\Spec \Lambda_{\cD'}\subset 
\Spec \Lambda_{\cD}$ induces the embedding
$\Char_k(\cD')\subset\Char_k(\cD)$ (cf. \cite{charvar}); 
any irreducible component of $V_k(\cD')$
is the intersection of an irreducible component of $V_k(\cD)$ with $\Lambda_{\cD'}$.

\begin{dfn}
Irreducible components of $V_k(\cD)$ contained in $\Lambda_{\cD'}$ for some 
curve $\cD' \subset \cD$ are called \emph{coordinate components} of $V_k(\cD)$. 
If an irreducible coordinate component~$V$ of $V_k(\cD')$ is also an 
irreducible component of $V_k(\cD)$, then $V$ is called a \emph{non-essential coordinate component},
otherwise it is called an \emph{essential coordinate component}.
\end{dfn}


See~\cite{ArtalCogolludo} for examples. A detailed discussion of more examples 
is done in sections~\ref{sec-fermat},~\ref{sec-ceva}, and~\ref{sec-zariski-pair}.

As shown in~\cite[Lemma~1.4.3]{charvar} (see also~\cite[Proposition~3.12]{dimca-pencils}), 
essential coordinate components must be zero dimensional.

\subsection{Alexander Invariant}
\label{sec-alex-inv}

In section~\ref{libg-var} the characteristic varieties of a finitely presented group $G$ are defined as
the zeroes of the Fitting ideals of the module $M:=G'/G''$ over $G/G'$. This module is referred to in the
literature as the \emph{Alexander invariant} of $G$. Note, however, that this is not the module represented
by the matrix of Fox derivatives called the \emph{Alexander module} of $G$.

Our purpose in this section is to briefly describe the Alexander invariant for fundamental groups of complements
of plane curves and give a method to obtain a presentation of such a module from a presentation of $G$. 
In order to do so, consider $G:=\pi_1(\PP^2\setminus \cD)$ the fundamental group of the curve $\cD$.
Without loss of generality one might assume that
\begin{enumerate}
 \item[$(Z1)$] $G/G'$ is a free group of rank $r$ generated by meridians $\gamma_1, \gamma_2, ..., \gamma_r$,
\end{enumerate}
then one has the following

\begin{lemma}[{\cite[Proposition~2.3]{ji-rybnikov}}]
Any group $G$ as above satisfying $(Z1)$ admits a presentation 
\begin{equation}
\label{eq-zar-pres}
\left\langle x_1,...,x_r,y_1,...,y_s : R_1(\bar x,\bar y)=...=R_m(\bar x,\bar y)=1 \right\rangle,
\end{equation}
where $\bar x:=\{x_1,...,x_r\}$ and $\bar y:=\{y_1,...,y_s\}$ satisfying:
\begin{enumerate}
 \item[$(Z2)$] 
$\ab (x_i)=\gamma_i$, $\ab (y_j)=0$, and $R_k$ can be written in terms of $\bar y$ and 
$x_k [x_i,x_j] x_k^{-1}$, where $[x_i,x_j]$ is the commutator of $x_i$ and $x_j$.
\end{enumerate}
\end{lemma}
A presentation satisfying $(Z2)$ is called a \emph{Zariski presentation} of~$G$.

From now on we will assume $G$ admits a Zariski presentation as in~\eqref{eq-zar-pres}.
In order to describe elements of the module $M$ it is sometimes convenient to see $\ZZ[G/G']$ as
the ring of Laurent polynomials in $r$ variables $\ZZ[t_1^{\pm 1},...,t_r^{\pm 1}]$, where $t_i$
represents the action induced by $\gamma_i$ on $M$ as a multiplicative action, that is,
\begin{equation}
\label{eq-action-M}
t_i g\ \eqmap{M}\ x_i g x_i^{-1} 
\end{equation}
for any $g\in G'$.

\begin{remark}
\mbox{}
\begin{enumerate}
 \item One of course needs to convince oneself that action~\eqref{eq-action-M} is independent, up to
an element of $G''$, of the representative $x_i$ as long as $\ab (x_i)=\gamma_i$. This is an easy exercise.
 \item We denote by ``$\eqmap{M}$'' equalities that are valid in $M$.
\end{enumerate}
\end{remark}

\begin{example}
\label{exam-rel}
Note that 
\begin{equation}
\label{eq-comm-M}
[xy,z]\ \eqmap{M}\ [x,z]+t_x[y,z],
\end{equation} 
where $x$, $y$, and $z$ are elements of $G$ and $t_x$ denotes
$\ab (x)$ in the multiplicative group. This is a consequence of the following
$$
[xy,z]=xyzy^{-1}x^{-1}z^{-1}=x(yzy^{-1}z^{-1})x^{-1} xzx^{-1}z^{-1}\ \eqmap{M}\ t_x[y,z]+[x,z].
$$
As a useful application of~\eqref{eq-comm-M} one can check that 
\begin{equation}
\label{eq-comm-M2}
[x^y,z]\ \eqmap{M}\ [x,z]+(t_z-1)[y,x],
\end{equation} 
where $x^y:=yxy^{-1}$.
\end{example}

Note that $x_{ij}:=[x_i,x_j]$, $1\leq i<j \leq r$ and $y_k$, $k=1,...,s$ are elements in $G'$, since
$\ab (x_{ij})=\ab (y_k)=0$. Therefore 
\begin{equation}
\label{eq-xijk}
x_k [x_i,x_j] x_k^{-1}\ \eqmap{M}\ t_k x_{i,j}
\end{equation} 
(see~\eqref{eq-action-M} and~$(Z2)$). Moreover,

\begin{prop}
\label{prop-M1}
For a group $G$ as above, the module $M$ is generated by $\bar x_{i,j}:=\{x_{ij}\}_{1\leq i<j \leq r}$ and 
$\bar y:=\{y_k\}_{k=1,...,s}$.
\end{prop}

\begin{example}
\label{exam-rel2}
The module $M$ is not freely generated by the set mentioned above, for instance, note that according 
to~$(Z2)$ and~\eqref{eq-xijk} any relation in $G$, say $R_i(\bar x,\bar y)=1$ (as in~\eqref{eq-zar-pres}) 
can be written (in $M$) in terms of $\bar \{x_{ij}\}$ and $\bar y$ as $\cR_i(\bar x_{ij},\bar y)$. In other 
words, $\cR_i(\bar x_{ij},\bar y)=0$ is a relation in $M$.
\end{example}

\begin{example}
\label{exam-rel3}
Even if $G$ were to be the free group $\FF_r$, $M$ would not be freely generated by
$\bar \{x_{ij}\}$ and $\bar y$. In fact,
\begin{equation}
\label{eq-J}
J(x,y,z):=(t_x-1)[y,z]+(t_y-1)[z,x]+(t_z-1)[x,y]\ \eqmap{M}\ 0
\end{equation}
for any $x$, $y$, $z$ in $G$. Using Example~\ref{exam-rel} repeatedly, one can check the following
\begin{equation}
\label{eq-J2}
[xy,z] =
\left\{
\array{l}
\eqmap{M}\ [x,z]+t_x[y,z]\\
= [y^{x^{-1}}x,z]\ \eqmap{M}\ [y^{x^{-1}},z]+ t_y [x,z]\ \eqmap{M}\ [y,z]-(t_z-1)[x,y]+t_y [x,z],\\
\endarray
\right.
\end{equation}
where $a^b=bab^{-1}$. The difference between both equalities results in $J(x,y,z)=0$.
Such relations will be referred to as \emph{Jacobian relations} of $M$.
\end{example}

A combination of Examples~\ref{exam-rel2} and~\ref{exam-rel3} gives in fact a presentation of $M$.

\begin{prop}[{\cite[Proposition~2.39]{ji-zariski}}]
\label{prop-M2}
The set of relations $\cR_1$,...,$\cR_m$ as described in Example~\ref{exam-rel2} and 
$J(i,j,k)=J(x_i,x_j,x_k)$ as described in Example~\ref{exam-rel3} is a complete system of relations for $M$.
\end{prop}

\begin{example}
Let $G=\FF_r$ be the free group in $r$ generators, for instance, the fundamental group of the complement to the 
union of $r+1$ concurrent lines. According to Propositions~\ref{prop-M1} and~\ref{prop-M2}, $M$ has a presentation
matrix $A_r$ of size ${{r}\choose {3}} \times {{r}\choose {2}}$ whose columns correspond to the generators 
$x_{ij}=[x_i,x_j]$ and whose rows correspond to the coefficients of the Jacobian relations 
$J(i,j,k)$, $1\leq i<j<k \leq r$. For instance, if $r=4$
$$
A_4:=
\left[
\begin{matrix}
(t_3-1) & -(t_2-1) & 0 & (t_1-1) & 0 & 0\\
(t_4-1) & 0 & -(t_2-1) & 0 & (t_1-1) & 0\\
0 & (t_4-1) & -(t_3-1) & 0 & 0 & (t_1-1)\\
0 & 0 & 0 & (t_4-1) & -(t_3-1) & (t_2-1)\\
\end{matrix}
\right].
$$
Such matrices have rank ${{r-1}\choose {2}}$ if $t_i\neq 1$ for all $i=1,...,r$, and hence the depth of
a non-coordinate character is $r-1$. On the other hand, for the trivial character $\bar 1$, the matrix $A_n$ 
has rank~0 and hence $\bar 1$ has depth ${{r}\choose {2}}$ (see Definitions~\ref{def-fitting} and~\ref{def-charvar} 
for details on the connection between the rank of $A_n$ and the depth of a character). 
\end{example}

\subsection{Orbicurves} 

As a general reference for orbifolds and orbifold fundamental groups 
one can use ~\cite{adem}, see
also~\cite{friedman,scott}. A brief description of what will be used here follows.

\begin{dfn}  
\label{def-global} 
An \emph{orbicurve} is a complex orbifold of dimension equal to one.
An orbicurve $\cC$ is called a \emph{global quotient} if there exists a finite group $G$ acting effectively on 
a Riemann surface $C$ such that $\cC$ is the quotient of $C$ by $G$ 
with the orbifold structure given by the stabilizers of the $G$-action on~$C$.
\end{dfn}

We may think of $\cC$ as a Riemann surface with a finite number of points $R:=\{P_1,...,P_s\}\subset \cC$ 
labeled with positive integers $\{m(P_1),...,m(P_s)\}$ 
(for global quotients those are the orders of stabilizers of 
action of $G$ on $C$). A neighborhood of a point $P\in \cC$ with 
$m(P)>0$ is the quotient of a disk (centered at~$P$) by 
an action of the cyclic group of order $m(P)$ (a rotation).
 
A small loop around $P$ is considered to be trivial in $\cC$ if its 
lifting in the above quotient map bounds a disk. Following this 
idea, orbifold fundamental groups can be defined as follows.

\begin{dfn}\label{dfn-group-orb}(cf.~\cite{adem,scott,friedman})
Consider an orbifold $\cC$ as above, then the \emph{orbifold fundamental group} of $\cC$ is
$$
\pi_1^\orb(\cC):=\pi_1(\cC\setminus\{P_1,\dots,P_s\})/\langle\mu_j^{m_j}=1\rangle
$$
where $\mu_j$ is a meridian of $P_j$ and $m_j:=m(P_j)$.
\end{dfn}

According to Remark~\ref{charvargroups} the Definition~\ref{def-charvar} can be applied 
to the case of finitely generated groups. In particular one defines the $k$-th 
characteristic variety $\Char_k(\cC)$ of an orbicurve $\cC$ as $\Char_k(\pi_1^\orb)$.
Therefore also the concepts of a character $\chi$ on $\cC$ and its depth are well defined.

\begin{example}
\label{exam-orbicurve}
Let us denote by $\PP^1_{m_1,...,m_s,k \infty}$ an orbicurve 
for which the underlying Riemann surface is $\PP^1$ with 
$k$ points removed and $s$ labeled points with labels $m_1,...,m_s$. 
If $k\geq 1$ (resp. $k\geq 2$) we also use the notation
$\CC_{m_1,...,m_s,(k-1) \infty}$ 
(resp. $\CC^*_{m_1,...,m_s,(k-2) \infty}$)
for 
$\PP^1_{m_1,...,m_s,k \infty}$.
We suppress specification of actual points on $\PP^1$. 
Note that
$$\pi_1^\orb(\PP^1_{m_1,...,m_s,k \infty})=
\begin{cases}
\ZZ_{m_1}(\mu_1)*...*\ZZ_{m_s}(\mu_s)*\ZZ*\dotsmap{k-1} *\ZZ & \text{if\ } k>0\\
{\ZZ_{m_1}(\mu_1)*...*\ZZ_{m_{s}}(\mu_{s}})/{\prod \mu_i} & \text{if\ } k=0\\
\end{cases}
$$
(here $\ZZ_m(\mu)$ denotes a cyclic group of order $m$ with a generator $\mu$). 
Note that a global quotient orbifold of $\PP^1\setminus \{nk \text{\ points}\}$
by the cyclic action of order $n$ on $\PP^1$ that fixes two points, that is, $[x:y]\mapsto [\xi_nx:y]$
(which fixes $[0:1]$ and $[1:0]$) is $\PP^1_{n,n,k \infty}$.

Interesting examples of elliptic global quotients occur for $\PP^1_{2,3,6,k \infty}$, 
$\PP^1_{3,3,3,k  \infty}$, and $\PP^1_{2,4,4,k  \infty}$, which are global orbifolds of 
elliptic curves $E\setminus \{6 k \text{\ points}\}$, $E\setminus \{3 k \text{\ points}\}$, 
and $E\setminus \{4 k \text{\ points}\}$ respectively, see~\cite{mordweil} for a study of 
the relationship between these orbifolds ($k=0$) and the depth of characters of fundamental 
groups of the complements to plane singular curves.
\end{example}

\begin{dfn}
A \emph{marking} on an orbicurve $\cC$ (resp. a quasiprojective variety $\cal X$) is a
non-trivial character of its orbifold fundamental group (resp. its fundamental group) of
positive depth $k$, that is, an element of $\Hom(\pi_1^\orb(\cC),\CC^*)$
(resp. $\Hom(\pi_1(\cal X),\CC^*)$) which is in $V_k(\cC)$ (resp. $V_k(\cX)$).

A \emph{marked orbicurve} is a pair $(\cC,\rho)$, where $\cC$ is an orbicurve and $\rho$
is a marking on $\cC$. Analogously, one defines a \emph{marked quasi-projective manifold}
as a pair $(\cX,\chi)$ consisting of a quasi-projective manifold $\cX$ and a marking on it.

A marked orbicurve $(\cC,\rho)$ is \emph{a global quotient} if $\cC$ is a global quotient of $C$, where
$C$ is a branched cover of $\cC$ associated with the unbranched cover of $\cC\setminus \{P_1,...,P_s\}$ 
corresponding to the kernel of $\pi_1(\cC\setminus \{P_1,...,P_s\})\to \pi_1^\orb (\cC) \rightmap{\rho} \CC^*$. 
In other words, the covering space in Definition~\ref{def-global} corresponds to the kernel of $\rho$.
\end{dfn}

\subsection{Orbifold pencils on quasi-projective manifolds}\label{quasiprojpencils}

\begin{dfn}
Let $\cX$ be a quasi-projective variety, $C$ be a quasi-projective curve, and $\cC$ an orbicurve
which is a global quotient of $C$. A \emph{global quotient orbifold pencil} is a map
$\phi: \cal X \rightarrow \cC$ such that there exists $\Phi: X_G \rightarrow C$ where $X_G$ is a
quasi-projective manifold endowed with an action of the group $G$ making the following diagram commute:
\begin{equation}\label{markedpencildef}
\begin{matrix} X_G & \buildrel \Phi \over \rightarrow & C \cr
              \downarrow   &  &  \downarrow \cr
              \cX  & \buildrel \phi \over \rightarrow & \cC \cr
\end{matrix}  
\end{equation} 
The vertical arrows in \eqref{markedpencildef} are the quotients by the action of $G$.

If, in addition, $(\cX,\chi)$ and $(\cC,\rho)$ are marked, then the global quotient orbifold pencil
$\phi: \cX \rightarrow \cC$ called \emph{marked} if $\chi=\phi^*(\rho)$.
We will refer to the map of pairs $\phi:(\cX,\chi)\to (\cC,\rho)$ as a
\emph{marked global quotient orbifold pencil} on $(\cX,\chi)$ with
 target~$(\cC,\rho)$.
\end{dfn}

\begin{dfn}
\label{def-indep}
Global quotient orbifold pencils $\phi_i:(\cal X,\chi) \rightarrow (\cC,\rho)$, $i=1,...,n$
are called \emph{independent} if the induced maps $\Phi_i: X_G \rightarrow C$
define $\ZZ[G]$-independent morphisms of modules
\begin{equation}
{\Phi_i}_*: H_1(X_G,\ZZ) \rightarrow H_1(C,\ZZ),
\end{equation}
that is, independent elements of the $\ZZ[G]$-module $\Hom_{\ZZ[G]}(H_1(X_G,\ZZ),H_1(C,\ZZ))$.

In addition, if $\bigoplus {\Phi_i}_*: H_1(X_G,\ZZ) \rightarrow H_1(C,\ZZ)^n$ is surjective we say that
the pencils $\phi_i$ are \emph{strongly independent}.
\end{dfn}

\begin{remark}
Note that if either $n=1$ or $H_1(C,\ZZ)=\ZZ[G]$, then independence is equivalent to strong independence
(this is the case for Remark~\ref{rem-main-thm}\ref{rem-non-orbifold} and 
Theorem~\ref{thm-main}\ref{thm-main-part2}).
\end{remark}

\subsection{Structure of characteristic varieties (revisited)}

The following are relevant improvements or additions to Theorem~\ref{arapurath}:

\begin{theorem}[\cite{nonvanishing,acm-charvar-orbifolds}]
The isolated zero-dimensional characters of $V_k(\cD)$ are torsion characters of $\Char(\cD)$.
\end{theorem}

In~\cite[Theorem~3.9]{dimca-pencils} (see also~\cite{Dimca-Papadima-Suciu-formality}) there is a 
description of one-dimensional components $\chi f^*H^1(C,\CC^*)\subset \Char_k(\cX)$
mentioned in Theorem~\ref{arapurath} and most importantly, of the order of $\chi$ in terms of multiple
fibers of the rational pencil $f$.

In~\cite{charvar}, an algebraic method is described to detect the irregularity of abelian covers 
of $\PP^2$ ramified along $\cD$. This method is very useful to 
compute \emph{non-coordinate components} of $V_k(\cD)$
independently of a presentation of the fundamental group of the complement $\cX$ of $\cD$.

Theorem~\ref{thm-main} (see~\cite{acl-depth}) has~\cite[Theorem~3.9]{dimca-pencils} as a consequence,
but uses the point of view of orbifold pencils. Using this result also the zero-dimensional components 
can be detected (in particular essential coordinate components) and in some cases characterized 
(see section~\ref{sec-ceva}).

Another improvement of Theorem~\ref{arapurath} was given in~\cite{acm-charvar-orbifolds} were the 
point of view of orbifolds was first introduced as follows:

\begin{theorem}[\cite{acm-charvar-orbifolds}]
\label{thmprin}
Let $\cX$ be a smooth quasi-projective variety. Let $V$ be an irreducible component of $V_k(\cX)$. 
Then one of the two following statements holds:
\begin{enumerate}
\enet{\rm(\arabic{enumi})}
\item\label{thmprin-orb} 
There exists an orbicurve $\cC$, a surjective orbifold morphism $\rho:X\to \cC$ and 
an irreducible component $W$ of $V_k(\pi_1^{\text{\rm orb}}(\cC))$ such that $V=\rho^*(W)$.
\item\label{thmprin-tors} $V$ is an isolated torsion point not of type~\ref{thmprin-orb}.
\end{enumerate}
\end{theorem}

One has the following consequences from~\ref{thm-main}\ref{thm-main-part2} that allows us to characterize
certain elements of $V_k(\cD)$:

\begin{cor}
\label{rem-main-thm}
Let $(\cX,\chi)$ be a marked complement of $\cD$. Then possible targets for marked orbifold pencils are
$(\cC,\rho)$ with $\cC=\PP^1_{m_1,...,m_s,k \infty}$ (see Example~\ref{exam-orbicurve}).
Assume that there are $n$ strongly independent marked orbifold pencils with such a fixed target $(\cC,\rho)$.
Then,
\begin{enumerate}
\enet{\rm(\arabic{enumi})}
 \item\label{rem-non-orbifold}
In case $\cC$ has no orbifold points, that is $s=0$, the character $\chi$ 
belongs to a positive dimensional component $V$ of $\Char(\cX)$ containing the trivial character.
In this case, $d(\chi)= \dim V-1=n-2$.
 \item\label{rem-torsion-2}
In case $\chi$ is a character of order two, there is a unique marking on $\cC=\CC_{2,2}$ and $d(\chi)$ 
is the maximal number of strongly independent orbifold pencils with target~$\cC$.
 \item\label{rem-elliptic}
In case $\chi$ has torsion 3,4, or 6, there is a unique marking on $\cC=\PP^1_{3,3,3}$, $\cC=\PP^1_{2,4,4}$,
or $\cC=\PP^1_{2,3,6}$ respectively and $d(\chi)$ is the maximal number of strongly independent orbifold pencils 
with target~$\cC$.
\end{enumerate}
\end{cor}

Part~\ref{rem-non-orbifold} is a direct consequence of Theorem~\ref{arapurath} and
part~\ref{rem-elliptic} had already appeared in the context of Alexander polynomials in~\cite{mordweil}.

In section~\ref{sec-ceva} we will describe in detail examples of 
Corollary~\ref{rem-main-thm}\ref{rem-torsion-2} for line arrangements.

\subsection{Zariski pairs}

We will give a very brief introduction to Zariski pairs. For more details we refer to~\cite{ji-zariski} 
and the bibliography therein.

\begin{dfn}[\cite{kike}]
Two plane algebraic curves $\cD$ and $\cD'$ form a \emph{Zariski pair} if 
there are homeomorphic tubular neighborhoods of $\cD$ and $\cD'$, but
the pairs $(\PP^2,\cD)$ and $(\PP^2,\cD)$ are not homeomorphic.
\end{dfn}

The first example of a Zariski pair was given by Zariski~\cite{zariski}, who showed that the fundamental 
group of the complement to an irreducible  sextic (a curve of degree six) with six cusps on a conic is isomorphic 
to $\ZZ_2*\ZZ_3$ whereas the fundamental group of any other sextic with six cusps is $\ZZ_6$. 
This paved the way for intensive research aimed to understand the connection between the topology of 
$(\PP^2,\cD)$ and the position of the singularities of~$\cD$ (whether algebraically, geometrically, combinatorially...). 
This research has been often in the direction of a search for finer invariants of $(\PP^2,\cD)$.

Characteristic varieties (described above) and the Alexander polynomials
(i.e. the one variable version of the characteristic varieties), 
twisted polynomials~\cite{twisted}, generalized Alexander 
polynomials~\cite{oka-alexander,mordweil}, dihedral covers of $\cD$ (\cite{hiro}) among many others are 
examples of such invariants.

\begin{dfn}
If the Alexander polynomials $\Delta_\cD(t)$ and $\Delta_{\cD'}(t)$ coincide, then we say
$\cD$ and $\cD'$ form an \emph{Alexander-equivalent Zariski pair}.
\end{dfn}

In section~\ref{sec-zariski-pair} we will use Theorem~\ref{thm-main} to give an alternative proof 
that the curves in~\cite{ArtalCogolludo} 
Alexander-equivalent Zariski pair, without computing the fundamental group.

\section{Examples of characters of depth 3: Fermat Curves}
\label{sec-fermat}
Consider the following family of plane curves:

$$
\array{rcl}
\cF_n&:=&\{f_n:= x_1^n+x_2^n-x_0^n=0\},\\
\cL_1&:=&\{\ell_1:=x_0^n-x_2^n=0\},\\
\cL_2&:=&\{\ell_2:=x_0^n-x_2^n=0\}.
\endarray
$$

We will study the characteristic varieties of the quasi-projective manifolds 
$\cal X_n:=\PP^2\setminus \cD_n$, where $\cD_n:=\cF_n\cup \cL_1\cup \cL_2$, in light of the results given in the 
previous sections, in particular the essential torsion characters will be considered and 
their depth will be exhibited as the number of strictly independent orbifold pencils.

\subsection{Fundamental Group}
Note that $\cD_n$ is nothing but the preimage 
by the Kummer cover $[x_0:x_1:x_2]\mapstomap{\kappa_n} [x_0^n:x_1^n:x_2^n]$ of the following 
arrangement of three lines in general position given by the equation 
$$(x_0-x_1)(x_0-x_2)(x_0-x_1-x_2)=0.$$
Such a map ramifies along $\cB:=\{x_0x_1x_2=0\}$. We will compute the fundamental group of $\cal X_n$ as a quotient of 
the subgroup $K_n$ of $\pi_1(\PP^2\setminus \cL)$ associated with the Kummer cover, where 
\begin{equation}
\label{eq-ceva}
\cL:=\{x_0x_1x_2(x_0-x_2)(x_0-x_1)(x_0-x_1-x_2)=0\}
\end{equation}
is a Ceva arrangement. More precisely, the quotient is obtained as a factor of $K_n$ by the normal 
subgroup generated by the meridians of the ramification locus $\kappa_n^{-1}(\cB)$ in $\cal X_n$.

The fundamental group of the complement to the Ceva arrangement $\cL$ is given by the following 
presentation of~$G$.
\begin{equation}
\label{eq-G}
\langle e_0,...,e_5 : [e_1,e_2]=[e_3,e_5,e_1]=[e_3,e_4]=[e_5,e_2,e_4]=e_4e_3e_5e_2e_1e_0=1 \rangle
\end{equation}
where $e_i$ is a meridian of the component appearing in the $(i+1)$-th place in~\eqref{eq-ceva},
$[\alpha,\beta]$ denotes the commutator $\alpha\beta\alpha^{-1}\beta^{-1}$, and $[\alpha,\beta,\delta]$ denotes 
the triple of commutators $[\alpha\beta\delta,\alpha]$, $[\alpha\beta\delta,\beta]$, and $[\alpha\beta\delta,\delta]$
leading to a triple of relations in~\eqref{eq-G}.

In other to obtain~\eqref{eq-G} one can use the \emph{non-generic} Zariski-Van Kampen method on 
Figure~\ref{fig-ceva} (see~\cite[Section~1.4]{ji-zariski}). The dotted line $\ell$ represents a generic 
line where the meridians $e_0,...,e_5$ are placed (note that the last relation on~\eqref{eq-G} is the relation 
in the fundamental group of~$\ell \setminus (\cL\cap \ell)\approx \PP^1_{6\infty}$). The first two relations 
on~\eqref{eq-G} appear when moving the generic line around $\ell_1$. The third and fourth relations come from 
moving the generic line around~$\ell_4$.

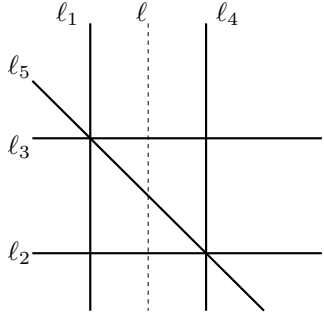
\begin{figure}
\label{fig-ceva}
\vspace*{14pt}
\begin{picture}(0,0)
 \put(13,110){$\ell_1$}
 \put(43,110){$\ell$}
 \put(73,110){$\ell_4$}
 \put(-5,60){$\ell_3$}
 \put(-5,90){$\ell_5$}
 \put(-5,20){$\ell_2$}
\end{picture}
 \includegraphics[scale=.3]{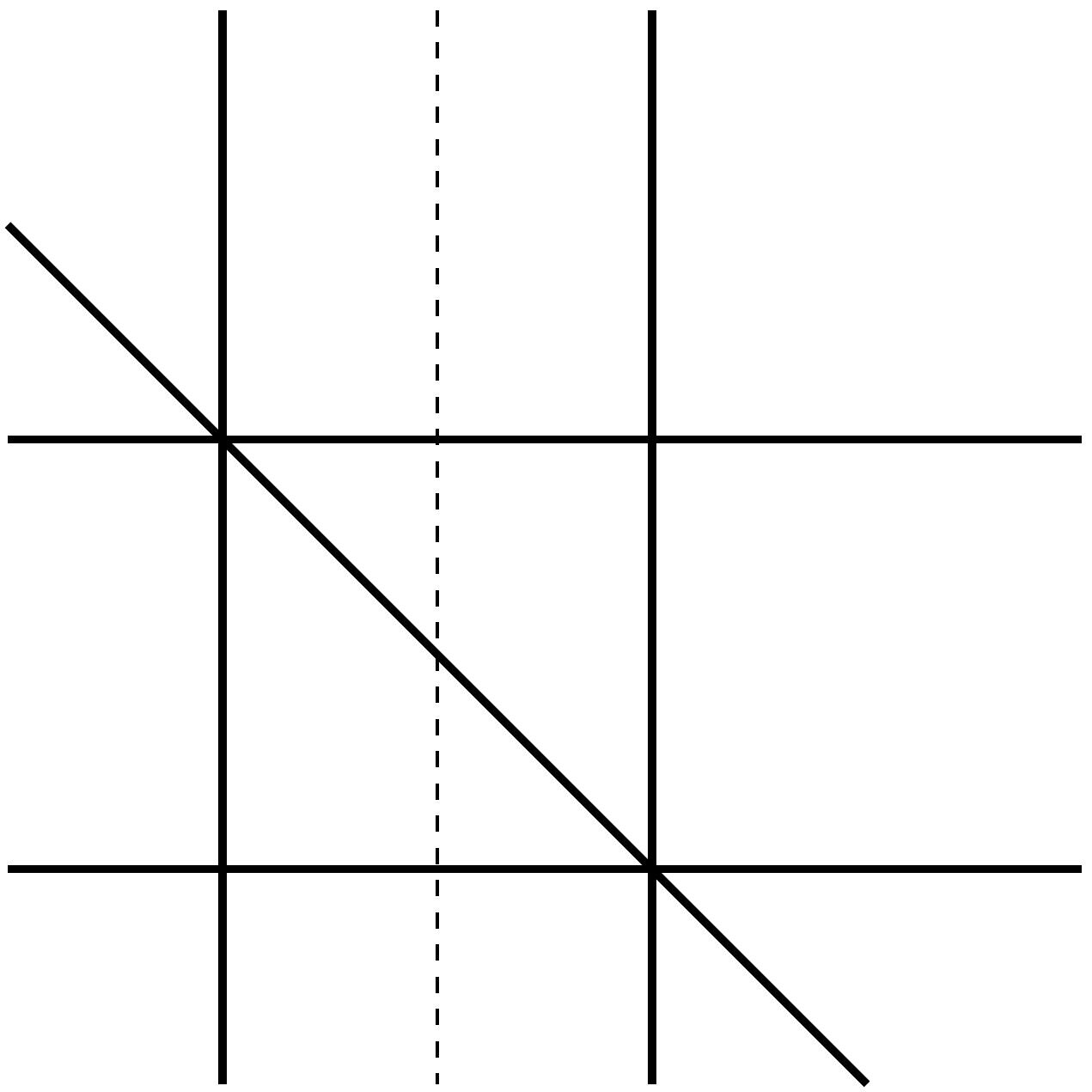}
\caption{Ceva arrangement}
\end{figure}

The fundamental group of the complement to $\cD_n\cup \cB$ is equal to the kernel $K_n$ of the epimorphism
\begin{equation}
\label{eq-epi}
\array{ccc}
G & \rightmap{\alpha} &\ZZ_n \times \ZZ_n \\
e_0 & \mapsto & (1,1)\\
e_1 & \mapsto & (1,0) \\
e_2 & \mapsto & (0,1) \\
e_3 & \mapsto & (0,0) \\
e_4 & \mapsto & (0,0) \\
e_5 & \mapsto & (0,0) \\
\endarray
\end{equation}
since it is the fundamental group of the abelian cover with covering transformations $\ZZ_n\times \ZZ_n$.
Therefore a presentation of the fundamental group of the complement to $\cD_n$ can be obtained 
by taking a factor of $K_n$ by the normal subgroup generated by $e^n_0$, $e^n_1$, and $e^n_2$ (which are the 
meridians to the preimages of the lines $x_0$, $x_1$, and $x_2$ respectively).
Using the Reidemeister-Schreier method (cf.~\cite{karras}) combined with the triviality of $e^n_0$, $e^n_1$, 
and $e^n_2$ one obtains the following presentation for $G_n:=\pi_1(\cX_n)$:
\begin{equation}
\label{eq-rels} 
G_n=\langle \ e_{3,i,j},e_{4,i,j},e_{5,i,j} : 
\begin{array}{cl}
(R1) & e_{3,i+1,j}=e^{-1}_{5,i,j}e_{3,i,j}e_{5,i,j},\\
(R2) & e_{4,i,j+1}=e^{-1}_{5,i,j}e_{4,i,j}e_{5,i,j},\\
(R3) & e_{5,i+1,j}=e^{-1}_{5,i,j}e^{-1}_{3,i,j}e_{5,i,j}e_{3,i,j}e_{5,i,j},\\
(R4) & e_{5,i,j+1}=e^{-1}_{5,i,j}e^{-1}_{4,i,j}e_{5,i,j}e_{4,i,j}e_{5,i,j},\\
(R5) & {[}e_{3,i,j},e_{4,i,j}{]}=1,\\
(R6) & \prod_{k=0}^{n-1}e_{4,k,k}e_{3,k,k}e_{5,k,k}=1
\end{array}
\ \rangle
\end{equation}
where $i,j\in \ZZ_n$ and 
$$e_{k,i,j}:=e_1^ie_2^je_ke_2^{-j}e_1^{-i},\ \ k=3,4,5.$$
As a brief description of the Reidemeister-Schreier method, we recall that the generators of $G_n$
are obtained from a set-theoretical section of $\alpha$ in~\eqref{eq-epi} (in our case $s:\ZZ_n\times \ZZ_n \to G$
is given by $(i,j)\mapsto e_1^ie_2^j$) as follows 
$$
s(i,j)\ e_k\ (\alpha(e_k)s(i,j))^{-1}.
$$
Thus the set $\{e_{k,i,j}\}$ above forms a set of generators of $G_n$. Finally a complete set of relations can be 
obtained by rewriting the relations of $G$ in~\eqref{eq-G} (and their conjugates by $s(i,j)$) in terms of the 
generators of the subgroup $G_n$.

\begin{example}
In order to illustrate the rewriting method we will proceed with the second relation of $G$ in~\eqref{eq-G}.
$$
\array{c}
s(i,j)[e_1,e_2]s(i,j)^{-1}=e_1^ie_2^j(e_3e_4e_3^{-1}e_4^{-1})e_2^{-j}e_1^{-i}=\\
(e_1^ie_2^je_3e_2^{-j}e_1^{-i})\ (e_1^ie_2^je_4e_2^{-j}e_1^{-i})\ (e_1^ie_2^je_3^{-1}e_2^{-j}e_1^{-i})\ (e_1^ie_2^je_4^{-1}e_2^{-j}e_1^{-i})=
[e_{3,i,j},e_{4,i,j}]
\endarray
$$
\end{example}

\subsection{Essential Coordinate Characteristic Varieties}
\label{sec-charvar}
Now we will discuss a presentation of $G_n'/G_n''$ as a module over $G_n/G'_n$, which will be referred to as 
$M_{\cD_n,\ab}$. For details we refer to section~\ref{sec-alex-inv}.
Note that $G_n/G'_n$ is isomorphic to $\ZZ^{2n}$ and is generated by 
the cycles $\gamma_5$, $\gamma_{3,j}$, $\gamma_{4,i}$, ($i,j\in \ZZ_n$) where
$\gamma_5=\ab(e_{5,i,j})$, $\gamma_{3,j}=\ab(e_{3,i,j})$, and $\gamma_{4,i}=\ab(e_{4,i,j})$
satisfying $n\gamma_5+\sum_j \gamma_{3,j} + \sum_i \gamma_{4,i}=0$\footnote{Recall that $\ab$ is 
the morphism of abelianization}.
Let $t_{5}$ (resp. $t_{3,j}$, $t_{4,i}$) be the generators of $G_n/G_n'$ viewed as a 
multiplicative group corresponding to the additive generators $\gamma_5$ 
(resp. $\gamma_{3,j}$, $\gamma_{4,i}$). The characteristic varieties of $G_n$ are contained~in 
$$
(\CC^*)^{2n} =
\spec \CC[t_5^{\pm 1},t_{3,i}^{\pm 1},t_{4,j}^{\pm 1}]/(t_5^n\prod_j t_{3,j} \prod_i t_{4,i}-1).
$$
As generators of $M_{\cD_n,\ab}$ we select commutators of the generators of $G_n$ as given in~\eqref{eq-G}.
In order to do so, note that using relations $(R1)-(R4)$ in~\eqref{eq-rels}, a presentation of $G_n$ can be given in 
terms the $2n+1$ generators $e_5:=e_{5,0,0}$, $e_{3,j}:=e_{3,0,j}$, and $e_{4,i}:=e_{4,i,0}$. Hence,
by Proposition~\ref{prop-M1}, $M_{\cD_n,\ab}$ is generated by the ${2n+1}\choose{2}$ commutators 
\begin{equation}
\label{eq-comm}
\{[e_{5},e_{3,j}],[e_{5},e_{4,i}],[e_{4,i},e_{3,j}],[e_{4,i_1},e_{4,i_2}],[e_{3,j_1},e_{3,j_2}]\}_{i_*,j_*\in \ZZ_n},
\end{equation}
as a $\CC[\ZZ[t_1^{\pm 1},...,t_4^{\pm 1},t_5^{\pm 1}]]$-module. Also, according to Proposition~\ref{prop-M2},
a complete set of relations of $M_{\cD_n,\ab}$ is given by rewriting the following relations 
\begin{equation}
\label{eq-relsM} 
\array{lllll}
(M1) & {[}\prod_{i=0}^{n-1} e_{5,i,j},e_{3,j}] & = & 0&\\
(M2) & {[}\prod_{i=0}^{n-1} e_{5,i,j},e_{4,i}] & = & 0&\\
(M3) & {[}e_{5,i,j+1},e_{3,i,j+1}e_{5,i,j+1}]e_{5,i,j+1}^{-1}&=&
{[}e_{5,i+1,j},e_{4,i+1,j}e_{5,i+1,j}]e_{5,i+1,j}^{-1}&\\
(M4) & \prod_{i=0}^{n-1} e_{4,i,i}e_{3,i,i}e_{5,i,i} & = &0&\\
\endarray
\end{equation}
in terms of commutators~\eqref{eq-comm} and by the Jacobian relations:
\begin{equation}
\label{eq-jacobi} 
\array{lll}
(t_{3,j}-1)[e_{5},e_{4,i}]+(t_{4,i}-1)[e_{3,j},e_{5}]+(t_{5}-1)[e_{4,i},e_{3,j}] & = & 0,\\
(t_{3,j_1}-1)[e_{5},e_{3,j_2}]+(t_{3,j_2}-1)[e_{3,j_1},e_{5}]+(t_{5}-1)[e_{3,j_2},e_{3,j_1}] & = & 0,\\
(t_{4,i_1}-1)[e_{5},e_{4,i_2}]+(t_{4,i_2}-1)[e_{4,i_1},e_{5}]+(t_{5}-1)[e_{4,i_2},e_{4,i_1}] & = & 0,\\
... & \\
\endarray
\end{equation}
In order to rewrite relations $(M1)-(M4)$ one needs to use~\eqref{eq-comm} repeatedly. In what follows, we
will concentrate on the characters of $\Char(\cD_n)$ contained in the coordinate axes $t_{3,j}=t_{4,i}=1$.
Computations for the general case can also be performed, but are more technical and tedious.

Since we are assuming $t_{3,j}=t_{4,i}=1$, and $t_5\neq 1$, relations in~\eqref{eq-jacobi} 
become $[e_{4,i},e_{3,j}]=[e_{3,j_2},e_{3,j_1}]=[e_{4,i_2},e_{4,i_1}]=0$ and hence $(R5)$ in~\eqref{eq-rels} become
redundant. A straightforward computation gives the following matrix where each line is a relation from~\eqref{eq-relsM} 
written in terms of the commutators $\{[e_5,e_{3,i}],[e_5,e_{4,i}]\}_{i\in \ZZ_n}$.
$$
A_n:=
{\tiny{
\left[
\begin{array}{cccccc|cccccc}
\phi_n & 0 & 0 & ... & 0 & 0 & 0 & 0 & 0 & ... & 0 & 0\\
0 & \phi_n & 0 & ... & 0 & 0 & 0 & 0 & 0 & ... & 0 & 0\\
&&& \vdots &&&&&& \vdots &&\\
0 & 0 & 0 & ... & 0 & \phi_n & 0 & 0 & 0 & ... & 0 & 0\\
\hline
\hline
0 & 0 & 0 & ... & 0 & 0 & \phi_n & 0 & 0 & ... & 0 & 0\\
0 & 0 & 0 & ... & 0 & 0 & 0 & \phi_n & 0 & ... & 0 & 0\\
&&& \vdots &&&&&& \vdots &&\\
0 & 0 & 0 & ... & 0 & 0 & 0 & 0 & 0 & ... & 0 & \phi_n\\
\hline
\hline
1 & -1 & 0 & ... & 0 & 0 & 1 & -1 & 0 & ... & 0 & 0\\
1 & -1 & 0 & ... & 0 & 0 & 0 & t & -t & ... & 0 & 0\\
&&& \vdots &&&&&& \vdots &&\\
1 & -1 & 0 & ... & 0 & 0 & 0 & 0 & 0 & ... & t^{n-2} & -t^{n-2}\\
1 & -1 & 0 & ... & 0 & 0 & -t^{n-1} & 0 & 0 & ... & 0 & t^{n-1}\\
\hline
&&& \vdots &&&&&& \vdots &&\\
\hline
0 & 0 & 0 & ... & t^{n-2} & -t^{n-2} & 1 & -1 & 0 & ... & 0 & 0\\
0 & 0 & 0 & ... & t^{n-2} & -t^{n-2} & 0 & t & -t & ... & 0 & 0\\
&&& \vdots &&&&&& \vdots &&\\
0 & 0 & 0 & ... & t^{n-2} & -t^{n-2} & 0 & 0 & 0 & ... & t^{n-2} & -t^{n-2}\\
0 & 0 & 0 & ... & t^{n-2} & -t^{n-2} & -t^{n-1} & 0 & 0 & ... & 0 & t^{n-1}\\
\hline
\hline
1 & t & t^2 & ... & t^{n-2} & t^{n-1} & 1 & t & t^2 & ... & t^{n-2} & t^{n-1}\\
\end{array}
\right]
}}
$$
More precisely, the first (resp. second) block of $A_n$ corresponds to the $n$ relations given in 
$(M1)$ (resp. $(M2)$) of~\eqref{eq-relsM}, $\phi_n:=\frac{t^n-1}{t-1}$, and $t=t_5$. The following 
$n$~blocks of $A_n$ (between double horizontal lines) correspond to the $n^2$ relations given 
in $(M3)$ of~\eqref{eq-relsM}. Note that the last row of each of these blocks is a consequence of
the remaining $n-1$ rows. The last block corresponds to the relation given in $(M4)$ of~\eqref{eq-relsM}. 

\begin{example}
In order to illustrate $A_n$ we will show how to rewrite the first relation for $n=3$, that is,
$$[e_{5,0,j}e_{5,1,j}e_{5,2,j},e_{3,j}]\ \eqmap{M}\ \phi_n [e_5,e_{3,j}].$$
Using~\eqref{eq-comm-M} one has
$$
[e_{5,0,j}e_{5,1,j}e_{5,2,j},e_{3,j}]\ \eqmap{M}\ 
[e_{5,0,j},e_{3,j}]+t[e_{5,1,j},e_{3,j}]+t^2[e_{5,0,j},e_{3,j}].
$$
Therefore, it is enough to show that $[e_{5,i,j},e_{3,j}]=[e_5,e_{3,j}]$.
Note that $e_{5,i,j}$ is a conjugate of $e_5$ (using $(R4)$ and $(R4)$), hence, by~\eqref{eq-comm-M2} one obtains 
$[e_{5,i,j},e_{3,j}]=[e_5,e_{3,j}]$ (since we are assuming $t_{3,j}=1$).
\end{example}

Also note that, performing row operations, one can obtain the following equivalent matrix
$$
B_n:=
{\tiny{
\left[
\begin{array}{cccccc|cccccc}
\phi_n & 0 & 0 & ... & 0 & 0 & 0 & 0 & 0 & ... & 0 & 0\\
0 & \phi_n & 0 & ... & 0 & 0 & 0 & 0 & 0 & ... & 0 & 0\\
&&& \vdots &&&&&& \vdots &&\\
0 & 0 & 0 & ... & 0 & \phi_n & 0 & 0 & 0 & ... & 0 & 0\\
\hline
0 & 0 & 0 & ... & 0 & 0 & \phi_n & 0 & 0 & ... & 0 & 0\\
0 & 0 & 0 & ... & 0 & 0 & 0 & \phi_n & 0 & ... & 0 & 0\\
&&& \vdots &&&&&& \vdots &&\\
0 & 0 & 0 & ... & 0 & 0 & 0 & 0 & 0 & ... & 0 & \phi_n\\
\hline
1 & -1 & 0 & ... & 0 & 0 & 1 & -1 & 0 & ... & 0 & 0\\
1 & -1 & 0 & ... & 0 & 0 & 0 & t & -t & ... & 0 & 0\\
&&& \vdots &&&&&& \vdots &&\\
1 & -1 & 0 & ... & 0 & 0 & 0 & 0 & 0 & ... & t^{n-2} & -t^{n-2}\\
0 & t & -t & ... & 0 & 0 & 0 & 0 & 0 & ... & t^{n-2} & -t^{n-2}\\
&&& \vdots &&&&&& \vdots &&\\
0 & 0 & 0 & ... & t^{n-2} & -t^{n-2} & 0 & 0 & 0 & ... & t^{n-2} & -t^{n-2}\\
\hline
1 & t & t^2 & ... & t^{n-2} & t^{n-1} & 1 & t & t^2 & ... & t^{n-2} & t^{n-1}\\
\end{array}
\right]
}}
$$
Finally, one can write the presentation matrix $B_n$ in terms of the basis 
$$\{[e_5,e_{3,i}] - [e_5,e_{3,i+1}], [e_5,e_{3,n-1}] , [e_5,e_{4,i}] - [e_5,e_{4,i+1}], [e_5,e_{4,n-1}] \}_{i=0,...,n-2}$$
resulting in
$$
{\tiny{
\left[
\begin{array}{cccccc|cccccc}
\phi_n & \phi_n & 0 & ... & 0 & 0 & 0 & 0 & 0 & ... & 0 & 0\\
0 & \phi_n & \phi_n & ... & 0 & 0 & 0 & 0 & 0 & ... & 0 & 0\\
&&& \vdots &&&&&& \vdots &&\\
0 & 0 & 0 & ... & \phi_n & \phi_n & 0 & 0 & 0 & ... & 0 & 0\\
0 & 0 & 0 & ... & 0 & \phi_n & 0 & 0 & 0 & ... & 0 & 0\\
\hline
0 & 0 & 0 & ... & 0 & 0 & \phi_n & \phi_n  & 0 & ... & 0 & 0\\
0 & 0 & 0 & ... & 0 & 0 & 0 & \phi_n & \phi_n  & ... & 0 & 0\\
&&& \vdots &&&&&& \vdots &&\\
0 & 0 & 0 & ... & 0 & 0 & 0 & 0 & 0 & ... & \phi_n & \phi_n\\
0 & 0 & 0 & ... & 0 & 0 & 0 & 0 & 0 & ... & 0 & \phi_n\\
\hline
1 & 0 & 0 & ... & 0 & 0 & 1 & 0 & 0 & ... & 0 & 0\\
1 & 0 & 0 & ... & 0 & 0 & 0 & t & 0 & ... & 0 & 0\\
1 & 0 & 0 & ... & 0 & 0 & 0 & 0 & t^2 & ... & 0 & 0\\
&&& \vdots &&&&&& \vdots &&\\
1 & 0 & 0 & ... & 0 & 0 & 0 & 0 & 0 & ... & t^{n-2} & 0\\
0 & t & 0 & ... & 0 & 0 & 0 & 0 & 0 & ... & t^{n-2} & 0\\
0 & 0 & t^2 & ... & 0 & 0 & 0 & 0 & 0 & ... & t^{n-2} & 0\\
&&& \vdots &&&&&& \vdots &&\\
0 & 0 & 0 & ... & t^{n-2} & 0 & 0 & 0 & 0 & ... & t^{n-2} & 0\\
\hline
\phi_1 & \phi_2 & \phi_3 & ... & \phi_{n-1} & \phi_n & \phi_1 & \phi_2 & \phi_3 & ... & \phi_{n-1} & \phi_n\\
\end{array}
\right]
}}
$$
One can use the units in the third block to eliminate columns, leaving the equivalent matrix
$$
{\tiny{
\left[
\begin{array}{cccccc|cccccc}
\phi_n & 0 & 0 & ... & 0 & 0 & 0 & 0 & 0 & ... & 0 & 0\\
0 & 0 & 0 & ... & 0 & 0 & 0 & 0 & 0 & ... & -\phi_n & 0\\
0 & 0 & 0 & ... & 0 & \phi_n & 0 & 0 & 0 & ... & 0 & 0\\
\hline
0 & 0 & 0 & ... & 0 & 0 & 0 & 0 & 0 & ... & 0 & \phi_n\\
\hline
1 & 0 & 0 & ... & 0 & 0 & 1 & 0 & 0 & ... & 0 & 0\\
1 & 0 & 0 & ... & 0 & 0 & 0 & 0 & 0 & ... & t^{n-2} & 0\\
\end{array}
\right]
}}
\cong
{\tiny{
\left[
\begin{array}{c|ccc}
0 & -\phi_n & 0 & 0\\
0 & 0 & -\phi_n & 0\\
\phi_n & 0 & 0 & 0\\
\hline
0 & 0 & 0 & \phi_n\\
\hline
0 & -1 &  t^{n-2} & 0\\
\end{array}
\right].
}}
$$
Finally, a last combination of row operations using the units to eliminate columns results in
$$
{\tiny{
\left[
\begin{array}{c|ccc}
0 & 0 & -\phi_n t^{n-2} & 0\\
0 & 0 & -\phi_n & 0\\
\phi_n & 0 & 0 & 0\\
\hline
0 & 0 & 0 & \phi_n\\
\hline
0 & -1 & t^{n-2} & 0\\
\end{array}
\right]
}}
\cong
{\tiny{
\left[
\begin{array}{c|cc}
0 & -\phi_n t^{n-2} & 0\\
0 & -\phi_n & 0\\
\phi_n & 0 & 0\\
\hline
0 & 0 & \phi_n
\end{array}
\right]
}}
\cong
{\tiny{
\left[
\begin{array}{c|cc}
0 & \phi_n & 0\\
\phi_n & 0 & 0\\
\hline
0 & 0 & \phi_n
\end{array}
\right].
}}
$$
Hence the $n-1$ non-trivial torsion characters $\chi_n^i:=(\xi_n^i,1,...,1)$, $i=1,...,n$ 
belong to $\Char(\cD_n)$ and have depth 3, that is, $\chi_n^i\in V_3(\cD_n)$.

\subsection{Marked Orbifold Pencils}
By Theorem~\ref{thm-main}\ref{thm-main-part1} there are at most three strongly independent marked 
orbifold pencils from the marked variety $(\cX_n,\chi_n)$. Our purpose is to explicitly show such
three strongly independent pencils. Note that 
\begin{equation}
\label{eq-orb-maps}
\array{cccl}
j_k: & \PP^2\setminus (\cF_n\cup \cL_k) & \to & \CC^*_{n}
=\PP^1_{(n,[1:0]),(\infty,[0:1]),(\infty,[1:1])} \\
& [x:y:z] & \mapsto & [f_n:x_k^n],
\endarray
\end{equation}
for $j=1,2$ are two natural orbifold pencils coming from the $n$-ordinary points of $\cF_n$
coming form the triple points of the Ceva arrangement $\cL$
which are in $\cB$.
Consider the marked orbicurve $(\CC_{n,n},\rho_n)$, where $\rho_n=(\xi_n,1)$, the first coordinate corresponds
to the image of a meridian $\mu_1$ around $[0:1]\in \PP^1_{(n,[0:1]),(n,[1:0]),(\infty,[1:1])}$ and the 
second coordinate corresponds to the image of a meridian $\mu_2$ around $[1:0]$ (note that 
$\pi_1^\orb(\CC_{n,n})=\ZZ_n(\mu_1) * \ZZ_n(\mu_2)$).

In order to obtain marked orbifold pencils with target $(\CC_{n,n},\rho_n)$ one simply 
considers the following composition, where $i_k$ and $j_k$ are inclusions
$$\psi_k:\cX_n \injmap{i_k} \PP^2\setminus (\cF_n \cup \cL_k) \rightmap{j_k} 
\PP^1_{(n,[1:0]),(\infty,[0:1]),(\infty,[1:1])} \injmap{i} \PP^1_{(n,[0:1]),(n,[1:0]),(\infty,[1:1])}.$$
Such pencils are clearly marked global quotient orbifold pencils from $(\cX_n,\chi_n)$ to $(\CC_{n,n},\rho_n)$,
where $(\CC_{n,n},\rho_n)$ is the marked quotient of $C_n:=\PP^1\setminus \{[\xi_n^j:1]\}_{j\in \ZZ_n}$ by the cyclic 
action $[x:y]\mapsto [\xi_nx:y]$. The resulting commutative diagrams are given by
\begin{equation}
\label{eq-diag1}
\array{rcl}
X_n & \rightmap{\Psi_k} & C_n\\
{[}x_0:x_1:x_2:w] & \mapsto & [w:x_k] \\
\downmap{\pi} & & \downarrow\\
\cX_n & \rightmap{\psi_k} & \CC_{n,n}\\
{[}x_0:x_1:x_2] & \mapsto & [f_n:x_k^n],
\endarray
\end{equation}
$k=1,2$, where $X_n$ is the smooth open surface given by 
$\{[x_0:x_1:x_2:w]\in \PP^3 \mid w^n=f_n\} \setminus \{f_n\ell_1\ell_2=0\}$.

Note that there is a third quasitoric relation involving all components of $\cD_n$, namely,
\begin{equation}
\label{eq-3}
f_nx_0^n + \ell_1\ell_2 = x_1^nx_2^n
\end{equation}
and hence a global quotient marked orbifold map
\begin{equation}
\label{eq-orb-maps3}
\array{cccl}
\psi_3: & \cX_n & \to & \CC_{n,n}
=\PP^1_{(n,[0:1]),(n,[1:0]),(\infty,[1:1])} \\
& [x:y:z] & \mapsto & [-f_nx_0^n:x_1^nx_2^n],
\endarray
\end{equation}
which gives rise to the following diagram
\begin{equation}
\label{eq-diag2}
\array{rcl}
X_n & \rightmap{\Psi_3} & C_n\\
{[}x_0:x_1:x_2:w] & \mapsto & [-wx_0:x_1x_2] \\
\downmap{\pi_n} & & \downarrow\\
\cX_n & \rightmap{\psi_k} & \CC_{n,n}\\
{[}x_0:x_1:x_2] & \mapsto & [-f_nx_0^n:x_1^nx_2^n].
\endarray
\end{equation}
Note that, when extending $\pi_n$ to a branched covering, the preimage of each line $\{\ell_{k,i}=0\}\subset \cL_k$ 
($k=1,2$) in $\cD_n$ ($\ell_{k,i}:=x_0-\xi_n^ix_k$) decomposes into $n$ irreducible components $\bigcup_{j\in \ZZ_n} \ell_{k,i,j}$ 
and thus allows to consider $\gamma_{k,i,j}$ ($k=1,2$, $i,j\in\ZZ_n$) meridians around each component of $\{\ell_{k,i,j}=0\}$. 
Also consider a meridian $\gamma_0$ around the preimage of $\cF_n$.

\begin{theorem}
The marked orbifold pencils $\psi_1$, $\psi_2$, and $\psi_3$ described above are strongly independent and hence they
form a maximal set of strongly independent pencils.
\end{theorem}

\begin{proof}
Consider $\Psi_{\e,*}:H_1(X_n;\ZZ) \to H_1(C_n;\ZZ)=\ZZ[\xi_n]$, $\e=1,2,3$ the three equivariant morphisms described above.
Using the commutative diagrams~\eqref{eq-diag1} and~\eqref{eq-diag2} one can easily see that
\begin{equation}
\label{eq-gammas}
\Psi_{\e,*} (\gamma_{k,i,j}) = \begin{cases} \xi_n^j & \text{ if\ } \e=k\in \{1,2\} \\
 \xi_n^{i+j} & \text{ if\ } k=3 \\ 0 & \text{otherwise} \end{cases}
\end{equation}
and
$$
\Psi_{\e,*} (\gamma_{0}) = 0
$$
and therefore $\Psi_{\e,*}$ are surjective $\ZZ[\xi_n]$-module morphisms.
Also note that $[\gamma_{k,i,j}]=\mu_n^j[\gamma_{k,i,0}] \in H_1(X_n;\ZZ)$. Consequently according to~\eqref{eq-gammas}
one has 
$$
\left(\Psi_{1,*}\oplus \Psi_{2,*}\oplus \Psi_{3,*}\right) (\gamma_{k,i,0})=
\begin{cases}
(1,0,\xi_n^i) & \text{ if\ } k=1 \\
(0,1,\xi_n^i) & \text{ if\ } k=2 \\
\end{cases}
$$
which implies that $\Psi_{1,*}\oplus \Psi_{2,*}\oplus \Psi_{3,*}$ is surjective.
After the discussion of section~\ref{sec-charvar}, since the depth of $\xi_n^i$ is three, the set of 
strongly independent pencils is indeed maximal.
\end{proof}

\section{Order Two Characters: \SC}
\label{sec-ceva}
From Theorem~\ref{thm-main}\ref{thm-main-part2}, for any order two character $\chi$ of depth $k$ in the characteristic
variety of the complement of a curve there exist $k$ independent pencils associated with $\chi$ whose target is a global 
quotient orbifold of type $\CC_{2,2}$. 

Interesting examples for $k>1$ of this scenario are the {\SC} arrangements $\Ceva(2,s)$, 
$s=1,2,3$ (or \emph{erweiterte Ceva} cf.~\cite[Section~2.3.J, pg.~81]{geraden}). Consider the following set of lines:
\begin{equation}
\array{ccc}
\array{c}
\ell_1:=x\\
\ell_2:=y\\
\ell_3:=z 
\endarray &
\array{c}
\ell_4:=(y-z)\\
\ell_5:=(x-z)\\
\ell_6:=(x-y)
\endarray &
\array{l}
\ell_7:=(x-y-z)\\
\ell_8:=(y-z-x)\\
\ell_9:=(z-x-y).
\endarray
\endarray
\end{equation}
The curve $\cC_6:=\left\{ \prod_{i=1}^6\ell_i=0\right\}$ is a realization of the Ceva arrangement $\Ceva(2)$ (a.k.a. 
braid arrangement or $B_3$-reflection arrangement). Note that this realization is different from the one used 
in section~\ref{sec-fermat}. The curve $\cC_7:=\left\{ \prod_{i=1}^7\ell_i=0 \right\}$ is the \SC\ arrangement $\Ceva(2,1)$ 
(a.k.a. a realization of the non-Fano plane). The curve $\cC_8:=\left\{ \prod_{i=1}^8\ell_i=0\right\}$ is the 
\SC\ arrangement $\Ceva(2,2)$ (a.k.a. a deleted $B_3$-arrangement). Finally, 
$\cC_9:=\left\{\prod_{i=1}^9\ell_i=0\right\}$ is the \SC\ arrangement $\Ceva(2,3)$.

The characteristic varieties of such arrangements of lines are well known 
(c.f~\cite{suciu-enumerative,suciu-translated,dimca-pencils}).
Such computations are done via a presentation of the fundamental group and using Fox derivatives.
In most cases (except for the simplest ones) the need of computer support is basically unavoidable.
In~\cite[Example~3.11]{dimca-pencils} there is an alternative calculation of the positive dimensional components 
of depth~1 via pencils.

Here we will give an interpretation via orbifold pencils of the characters of depth~2, which will account for the 
appearance of these components of the characteristic varieties independently of computation of the fundamental group.

\subsection{Ceva and \SC\ Pencils}
Note that $x(y-z)-y(x-z)+z(x-y)=0$ and hence 
$$
\array{cccc}
f_C: & \PP^2 & \to & \PP^1\\
& [x:y:z] & \mapsto & [\ell_1\ell_4:\ell_2\ell_5]
\endarray
$$
is a pencil of conics such that $(f_C^*([0:1])=\ell_1\ell_4,f_C^{*}([1:0])=\ell_2\ell_5,f_C^{-1}([1:1])=\ell_3\ell_6)$
(we will refer to it as the \emph{Ceva pencil}). Analogously 
$$x(y-z)(x-y-z)^2-y(x-z)(y-z-x)^2+z(x-y)(z-x-y)^2=0$$ 
and hence
$$
\array{cccc}
f_{SC}: & \PP^2 & \to & \PP^1\\
& [x:y:z] & \mapsto & [\ell_1\ell_4\ell_7^2:\ell_2\ell_5\ell_8^2]
\endarray
$$
is a pencil of quartics such that 
$(f_{SC}^{*}([0:1])=\ell_1\ell_4\ell_7^2,f_{SC}^{*}([1:0])=\ell_2\ell_5\ell_8^2,f_{SC}^{*}([1:1])=\ell_3\ell_6\ell_9^2)$ 
(we will refer to it as the \emph{\SC\ pencil}). 

\subsection{Characteristic Varieties of $\cC_i$, $i=6,7,8,9$}
We include the structure of the characteristic varieties of these curves for the reader's convenience. 
As reference for such computations 
see~\cite{suciu-enumerative,suciu-translated,falk-arrangements,cohen-suciu-characteristic,charvar,libgober-yuzvinsky-local}.

We will denote by $\cX_*$ the complement of the curve $\cC_*$ in $\PP^2$, for $*=6,7,8,9$.

\subsubsection{Arrangement $\cC_6$.}
The characteristic variety $\Char(\cC_6)$ consists of four non-essential coordinate components associated with the four 
triple points of $\cC_6$ (see Remark~\ref{rem-main-thm}\ref{rem-non-orbifold})\footnote{a.k.a. local components} 
and one essential component of dimension~2 and depth~1 given by the Ceva pencil 
$$\psi_6:=f_C|_{\cX_6}: \cX_6 \to \PP^1\setminus \{[0:1],[1:0],[1:1]\}.$$

\subsubsection{Arrangement $\cC_7$.}
The characteristic variety $\Char(\cC_7)$ consists of six (resp. four) non-essential coordinate components associated with 
the six triple points of $\cC_7$ (resp. four $\cC_6$-subarrangements) of dimension~2 and depth~1. In addition, there is 
one extra character of order two, namely,
$$\chi_{7}:=(1,-1,-1,1,-1,-1,1)$$ 
of depth~2.\footnote{the subscript 7 refers to the arrangement $\cC_7$. Similar notation will be used in the examples
that follow. A second subscript (when necessary) will be used to index the characters considered.}
In order to check the value of the depth, one needs to find all marked orbifold pencils in $(\cX_7,\chi_{7})$ 
of target $(\CC_{2,2},\rho)$ where $\rho:=(-1,-1)$ is the only possible non-trivial character of $\CC_{2,2}$. 
Two such independent pencils are the following, 
$$\psi_{7,1}:=f_C|_{\cX_7}: \cX_7 \to \PP^1\setminus \{[0:1],[1:0],[1:1]\} \to 
\PP^1_{(2,[1:0]),(2,[1:1]),(\infty [0:1])}$$
and 
$$\psi_{7,2}:=f_{SC}|_{\cX_7}: \cX_7 \to \PP^1_{(2,[1:0]),(2,[1:1]),(\infty [0:1])}.$$
This is the maximal number of independent pencils by Theorem~\ref{thm-main}.

\subsubsection{Arrangement $\cC_8$.}
The characteristic variety $\Char(\cC_8)$ consists of six (resp. five) non-essential coordinate components associated 
with the six triple points of $\cC_8$ (resp. four $\cC_6$-subarrangements) of dimension~2 and depth~1. In addition, 
there is one 3-dimensional non-essential coordinate component of depth~2 associated with its quadruple point
(see Remark~\ref{rem-main-thm}\ref{rem-non-orbifold}). 

Consider the following \SC\ pencil
$$\psi_{8,1}:=f_{SC}|_{\cX_8}: \cX_8 \to \PP^1_{(2,[1:1]),(\infty [0:1]),(\infty [1:0])}.$$
Computation of the induced map on the variety of characters shows that this map yields the only 
non-coordinate translated component of dimension~1 and depth~1 observed in the references above.
Finally, there are two characters of order two, namely,
$$
\array{cc}
\chi_{8,1}:=(1,-1,-1,1,-1,-1,1,1) & \text{\ and}\\
\chi_{8,2}:=(-1,1,-1,-1,1,-1,1,1) &
\endarray
$$
of depth~2. In order to check the value of the depth, one needs to find two marked orbifold pencils 
on $(\cX_8,\chi_{8,1})$ with target $(\CC_{2,2},\rho)$, where 
$$\CC_{2,2}:=\PP^1_{(2,[1:0]),(2,[1:1]),(\infty [0:1])}$$
and $\rho:=(-1,-1,1)$ is the only non-trivial character of $\CC_{2,2}$. 
Two such independent pencils can, for example, be given as follows
$$\psi_{8,2}:=f_C|_{\cX_8}: \cX_8 \to \PP^1\setminus \{[0:1],[1:0],[1:1]\} \to 
\PP^1_{(2,[1:0]),(2,[1:1]),(\infty [0:1])}$$
and 
$$\psi_{8,3}:=f_{SC}|_{\cX_8}: \cX_8 \to \PP^1_{(2,[1:1])} \setminus \{[1:0],[0:1]\} \to 
\PP^1_{(2,[1:0]),(2,[1:1]),(\infty [0:1])}.$$

\subsubsection{Arrangement $\cC_9$.}
The characteristic variety $\Char(\cC_9)$ consists of four (resp. eleven) non-essential coordinate components 
associated with the four triple points of $\cC_9$ (resp. eleven $\cC_6$-subarrangements), which have dimension~2 
and depth~1. In addition, there are three 3-dimensional non-essential coordinate components of depth~2 associated 
with the quadruple points of $\cC_9$. Consider the following \SC\ pencil 
$$\psi_{9,1}:=f_{SC}|_{\cX_9}: \cX_9 \to \PP^1 \setminus \{[1:0],[0:1],[1:1]\}.$$
Computations of the induced map on the variety of characters show that this pencil yields the only non-coordinate 
translated component of dimension~2 and depth~1 observed in the references above.

Finally, there are also three characters of order two 
$$
\array{cc}
\chi_{9,1}:=(-1,-1,1,-1,-1,1,1,1,1), &\\
\chi_{9,2}:=(-1,1,-1,-1,1,-1,1,1,1), & \text{\ and}\\
\chi_{9,3}:=(1,-1,-1,1,-1,-1,1,1,1)
\endarray
$$
of depth~2. In order to check the value of the depth, one needs to find two independent marked orbifold pencils on 
$(\cX_9,\chi_{9,1})$ with target $(\CC_{2,2},\rho)$ where $\CC_{2,2}:=\PP^1_{(2,[0:1]),(2,[1:0]),(\infty [1:1])}$ and
$\rho:=(-1,-1,1)$ is the only non-trivial character on $\CC_{2,2}$. 
Two such independent pencils can be given, for example, as follows
$$\psi_{9,2}:=f_C|_{\cX_9}: \cX_9 \to \PP^1\setminus \{[0:1],[1:0],[1:1]\} \to 
\PP^1_{(2,[0:1]),(2,[1:0]),(\infty [1:1])}$$
and 
$$\psi_{9,3}:=f_{SC}|_{\cX_9}: \cX_9 \to \PP^1\setminus \{[0:1],[1:0],[1:1]\} \to 
\PP^1_{(2,[0:1]),(2,[1:0]),(\infty [1:1])}.$$

\begin{remark}
Note that the depth $2$ characters in $\Char(\cC_8)$ and $\Char(\cC_9)$ lie in the intersection of
positive dimensional components and this fact forces them to have depth greater than~$1$, 
see~\cite[Proposition~5.9]{acm-charvar-orbifolds}.
\end{remark}

\subsection{Comments on Independence of Pencils}
\mbox{}
\begin{itemize}
 \item \textbf{Depth conditions on the target:}
First of all note that the condition on the target $(\cC,\rho)$ to have $d(\rho)>0$ is essential in the discussion 
above, i.e. pencils with target satisfying $d(\rho)=0$ may not contribute to the characteristic varieties. 
For instance, the space $\cX_6$ also admits several global quotient pencils coming from the {\SC} pencil, namely
$$\psi'_{6}:=f_{SC}|_{\cX_6}: \cX_6 \to \PP^1_{(2,[0:1]),(2,[1:0]),(2,[1:1])} \to \PP^1_{(2,[0:1]),(2,[1:0])}.$$
However, the orbifold $\PP_{2,2}$ is a global quotient orbifold whose orbifold fundamental group is abelian, so 
no non-trivial characters belong to its characteristic variety.

\item \textbf{Independence of Pencils.} Here is an explicit argument for independence of pencils for one of
the cases discussed in last section. 
Consider the pencils $\psi_{9,2}$ and $\psi_{9,3}$ described above as marked pencils from $(\cX_9,\chi_{9,1})$
having $(\CC_{2,2},\rho)$ as target. The marking produces the following commutative diagrams:
$$
\array{rcl}
X_{9,2} & \rightmap{\Psi_{9,2}} & C_2\\
{[}x:y:z:w] & \mapsto & [\ell_1\ell_4:w] \\
\downmap{\pi} & & \downmap{\tilde \pi}\\
\cX_9 & \rightmap{\psi_{9,2}} & \CC_{2,2}\\
{[}x:y:z] & \mapsto & [\ell_1\ell_4:\ell_2\ell_5],
\endarray
$$
and
$$
\array{rcl}
X_{9,2} & \rightmap{\Psi_{9,3}} & C_2\\
{[}x:y:z:w] & \mapsto & [\ell_1\ell_4\ell_7:w\ell_8] \\
\downmap{\pi} & & \downmap{\tilde \pi}\\
\cX_9 & \rightmap{\psi_{9,3}} & \CC_{2,2}\\
{[}x:y:z] & \mapsto & [\ell_1\ell_4\ell_7^2:\ell_2\ell_5\ell_8^2],
\endarray
$$
where $X_{9,2}$ is contained in $\{[x:y:z:w] \mid w^2=\ell_1\ell_4\ell_2\ell_5\}$, $C_2:=\PP^1\setminus \{[1:1],[1:-1]\}$ and 
$\tilde \pi$ is given by $[u:v]\mapsto [u^2:v^2]$.

Consider $\gamma_{i,k}$, $i=3,6,7,8,9$, $k=1,2$ the lifting of meridians around $\ell_i$ in $X_{9,2}$. 
Also denote by $\ZZ[\ZZ_2]$ the ring of deck transformations of $\tilde \pi$ as before, where $\ZZ_2$ acts 
by multiplication by $\xi_2=(-1)$.
Note that, as before $\Psi_{9,2}(\gamma_{3,k})=\Psi_{9,2}(\gamma_{3,k})=(-1)^k$ and 
$\Psi_{9,3}(\gamma_{4,k})=\Psi_{9,3}(\gamma_{4,k})=(-1)^{k+1}$. However, $\Psi_{9,2}(\gamma_{9,k})=0$ and
$\Psi_{9,3}(\gamma_{9,k})=(-1)^k$. Therefore $\psi_{9,2}$ and $\psi_{9,3}$ are
independent pencils of $(\cX_9,\chi_{9,1})$ with target $(\CC_{2,2},\rho)$.
\end{itemize}

\section{Curve Arrangements}
\label{sec-zariski-pair}
Consider the space $\cM$ of sextics with the following combinatorics:
\begin{enumerate}
\item $\cC$ is a union of a smooth conic $\cC_2$ and a quartic $\cC_4$;
\item $\Sing(\cC_4) = \{P, S\}$ where $S$ is a cusp of type $\mathbb A_4$ and $P$ is a node of type $\mathbb A_1$;
\item $\cC_2 \cap \cC_4 = \{S, R\}$ where $S$ is a $\mathbb D_7$ on $\cC$ and $R$ is a $\mathbb A_{11}$ on $\cC$.
\end{enumerate}
In~\cite{ArtalCogolludo} it is shown that $\cM$ has two connected components, say $\cM^{(1)}$ and $\cM^{(2)}$. 
The following are equations for curves  in each connected component:
\begin{equation*}
\array{rl}
f_6^{(1)}=f_2^{(1)}f_4^{(1)}:= & \left(  \left( y+3x \right) z+\frac{3y^2}{2} \right) \\
& \left( x^2{z}^{2}- \left( x{y}^{2}+\frac{15}{2}\,x^2y+\frac{9}{2}x^3 \right) 
z-3x\,{y}^{3}-\frac{9x^2y^2}{4}+\frac{y^4}{4} \right) 
\endarray
\end{equation*}
for $\cC_6^{(1)}\in \cM^{(1)}$ and 
\begin{equation*}
\array{c}
f_6^{(2)}=f_2^{(2)}f_4^{(2)}:= \left(  \left( y+\frac{x}{3} \right) z-\frac{y^2}{6}
 \right)
\left( x{z}^{2}- \left( x{y}^{2}+\frac{9x^2y}{2}+\frac{3x^3}{2}\right) z
+\frac{y^4}{4}+\frac{3x^2y^2}{4} \right) 
\endarray
\end{equation*}
for $\cC_6^{(2)}\in \cM^{(2)}$. 

The curves $\cC_6^{(1)}$ and $\cC_6^{(2)}$ form a Zariski pair since their fundamental groups are not isomorphic.
This cannot be detected by Alexander polynomials since both are trivial. In~\cite{ArtalCogolludo} the existence of 
an essential coordinate character of order two in the characteristic variety of $\cC_6^{(2)}$ was shown enough to 
distinguish both fundamental groups, since the characteristic variety of $\cC_6^{(1)}$ is trivial.

By Theorem~\ref{thm-main}\ref{thm-main-part2} this fact can also be obtained by looking at possible orbifold pencils.
Note that there exists a conic $\cQ:=\{q=0\}$ passing through $S$ and $R$ such that $(\cQ,\cC_4^{(1)})_S=4$, 
$(\cQ,\cC_4^{(2)})_S=5$, and $(\cQ,\cC_2^{(2)})_R=3$, $(\cQ,\cC_2^{(2)})_R=3$. Consider $L:=\{\ell=0\}$ the tangent 
line to $\cQ$ at $S$. One has the following list of multiplicities of intersection:
$$
\array{ll}
(\cQ,\cC_2^{(2)}+2L)_S=(\cQ,\cC_4^{(2)})_S=5 & (\cQ,\cC_2^{(2)}+2L)_R=(\cQ,\cC_4^{(2)})_R=3\\
(\cC_4^{(2)},2\cQ)_S=(\cC_4^{(2)},\cC_2^{(2)}+2L)_S=10 & (\cC_4^{(2)},2\cQ)_R=(\cC_4^{(2)},\cC_2^{(2)}+2L)_R=6\\
(\cC_2^{(2)},\cC_4^{(2)})_S=(\cC_2^{(2)},2\cQ)_S=2 & (\cC_2^{(2)},\cC_4^{(2)})_R=(\cC_2^{(2)},2\cQ)_R=6\\
(L,\cC_4^{(2)})_S=(L,2\cQ)_S=4 & (L,\cC_4^{(2)})_R=(L,2\cQ)_R=0.
\endarray
$$
By~\cite{ji-noether}, this implies that $(\cC_2^{(2)}+2L,\cC_2^{(2)},2\cQ)$ are members of a pencil of quartics.
In other words, there is a marked orbifold pencil from $\cC:=\PP^2\setminus \cC_6^{(2)}$ marked with $\chi:=(-1,1)$ to 
$\PP^1_{(2,[0,1]),(2,[1:0]),(\infty [1:1])}$ given by $[x:y:z]\mapsto [f_2^{(2)}\ell^2:q^2]$ whose target mark
is the character $\rho:=(-1,-1,1)$.

\end{document}